\documentclass[11pt]{amsart}
\usepackage{geometry, tikz-cd}                
\usepackage{graphicx}
\usepackage{amssymb}
\usepackage{epstopdf}
\usepackage{xcolor}
\usepackage{ulem}
\usepackage{cancel}
\usepackage[percent]{overpic}
\usepackage{stackengine}
\newcommand\xrowht[2][0]{\addstackgap[.5\dimexpr#2\relax]{\vphantom{#1}}}

\DeclareMathOperator\Log{Log}

\numberwithin{equation}{section}

\theoremstyle{plain}
\newtheorem{thm}{Theorem}[section]
\newtheorem{cor}[thm]{Corollary}
\newtheorem{prop}[thm]{Proposition}
\newtheorem{lem}[thm]{Lemma}

\newtheorem{rem}[thm]{Remark}

\newtheorem{definition}[thm]{Definition}

\makeatletter
\@namedef{subjclassname@2020}{%
  \textup{2020} Mathematics Subject Classification}
\makeatother

\makeatletter
\def\moverlay{\mathpalette\mov@rlay}
\def\mov@rlay#1#2{\leavevmode\vtop{%
   \baselineskip\z@skip \lineskiplimit-\maxdimen
   \ialign{\hfil$\m@th#1##$\hfil\cr#2\crcr}}}
\newcommand{\charfusion}[3][\mathord]{
    #1{\ifx#1\mathop\vphantom{#2}\fi
        \mathpalette\mov@rlay{#2\cr#3}
      }
    \ifx#1\mathop\expandafter\displaylimits\fi}
\makeatother

\newcommand{\bigcupdot}{\charfusion[\mathop]{\bigcup}{\cdot}}

\title{Combinatorial growth in the modular group}

\author{Ara Basmajian}
\address[Ara Basmajian]{The Graduate Center, CUNY \\ 365 Fifth Ave., N.Y., N.Y., 10016 and} 
 \address{Hunter College, CUNY \\ 695 Park Ave., N.Y., N.Y., 10065, USA}
\email{abasmajian@gc.cuny.edu}
 \thanks{The first author was supported by a PSC-CUNY Grant and a grant from the Simons foundation (359956, A.B.) The second author was supported in part by the Faculty Development Plan through the School of Science at Manhattan College.}
\author{Robert Suzzi Valli}
\address[Robert Suzzi Valli]{Manhattan College\\
             Riverdale, N.Y. 10471, USA}
\email{robert.suzzivalli@manhattan.edu}
\thanks{}
\begin{document}
\begin{abstract} We consider an exhaustion of the modular orbifold by compact subsurfaces and  show that   the growth rate, in terms of word length,  of  the  reciprocal   geodesics  on such   subsurfaces (so named  low lying reciprocal geodesics) converge    to the growth rate of the full set of reciprocal geodesics on the modular orbifold. We derive  a similar result 
for the low lying geodesics and their growth rate convergence to the growth rate of the full set of closed geodesics.

\end{abstract}
\maketitle


\section{Introduction}
Consider the modular surface;  that is, the $(2,3,\infty)$ triangle 
orbifold,  $S=\Bbb{H}/ PSL(2,\mathbb{Z})$.
A {\it reciprocal geodesic} on the modular surface is a closed geodesic that begins and ends at the order two cone point, traversing it in both directions.   Its lift is the  conjugacy class of a hyperbolic element with axis passing through an order two fixed point. Specifying the geometric length of a geodesic is equivalent to specifying the   absolute trace of such a hyperbolic element by way of the formula, 
$T_{\gamma}=2 \cosh \frac{\ell(\gamma)}{2}$.
Sarnak \cite{Sar} showed 
\begin{equation*}
|\{\gamma \text{ a primitive reciprocal geodesic with }   T_{\gamma} \leq T  \}| \sim   \frac{3}{8} T.
\end{equation*}
Let  $\mathcal{C} \subset S$ be the cusp with its natural horocycle boundary  of length one. The {\it depth}  of a point in    
$\mathcal{C}$   is its distance to the natural horocycle of length one. 
For $m$ a positive integer, we define the {\it $m$-thick  part}    of $S$,
denoted $S_m$,  to be $S$ with the points a depth larger than $\log \frac{m+1}{2}$ deleted. Thus the $m$-thick parts form a compact exhaustion of $S$.   We are interested in the reciprocal geodesics that lie in the $m$-thick part (so called $m$-low lying reciprocal geodesics).  See Figure \ref{fig:recgeod}.
Bourgain and Kontorovich \cite{B-K3}  showed (in our terminology) that for any $\epsilon >0$, there is an  $m >0$ so that
the number of  fundamental   reciprocal geodesics in the $m$-thick part having   absolute trace  $ \leq T$  has growth rate at least $T^{1- \epsilon}$.  In particular, as $\epsilon$ goes to zero, $m$ goes to infinity and we have a nested, increasing set of compact subsets that converge (in an appropriate sense) to the modular orbifold $S$.  Combined with the Sarnak result this shows that the growth rates of the low lying reciprocal geodesics converge (for each $T$, as $m \rightarrow \infty$) to the growth rate of the reciprocal geodesics. 
 Using the fact that $\mathbb{Z}_2\ast \mathbb{Z}_3$ is isomorphic to the modular group  we give a combinatorial analogue  of the above results using   word length, instead of absolute trace, with respect to the generators of the factors in $\mathbb{Z}_2\ast \mathbb{Z}_3$. For $\gamma$
a closed geodesic on $S$ we define its word length, denoted 
 $|\gamma |$, to be  the word length of one of its   lifts 
 $g \in PSL(2,\mathbb{Z})$, which is necessarily even.  Our  main results are
 
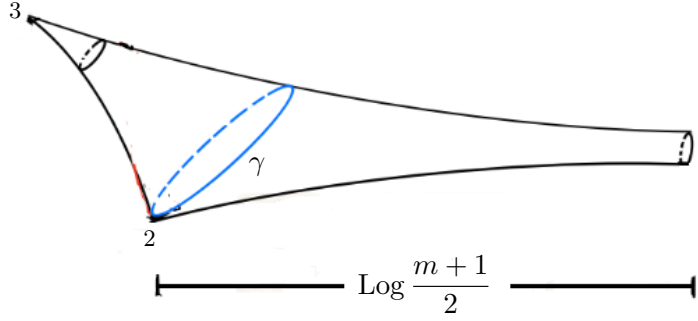
\begin{figure}
\begin{center}
\begin{overpic}[scale=.7]{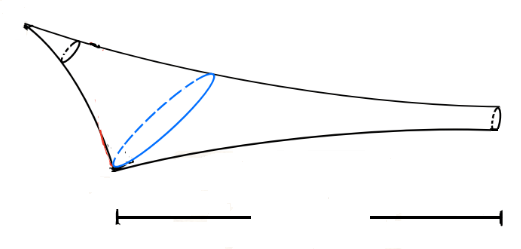}
\put(2,43){\footnotesize{$3$}}
\put(20.5,11.5){\footnotesize{$2$}}
\put(35,22){$\gamma$}
\put(50,5){$\Log \dfrac{m+1}{2}$}
\end{overpic}
\end{center}
\caption{A reciprocal geodesic $\gamma$ on $S_m$}
\label{fig:recgeod}
\end{figure}

\begin{thm} \label{thm:main}   

\leavevmode
\begin{enumerate}
\item $ |\{\gamma \text{ a primitive reciprocal geodesic with }
 |\gamma| \leq 2t\}| \sim  2^{\lfloor\frac{t}{2}\rfloor}$\\
\item
\resizebox{.9\hsize}{!}
{$|\{\gamma \text{ a primitive   reciprocal geodesic in $S_{m\geq2}$\,with }
 |\gamma| \leq 2t\}| \sim \left(\frac{\alpha_m}{2+(m+1)(\alpha_m-2)}\right)\alpha_m^{\lfloor\frac{t}{2}\rfloor}$}\\
\item
 $|\{\gamma \text{ a primitive closed geodesic with }
 |\gamma| \leq 2t\}| \sim  \frac{2^{t+1}}{t}$\\
\item
 $|\{\gamma \text{ a primitive closed geodesic in $S_{m\geq3}$ with }
 |\gamma| \leq 2t\}| \gtrsim
  \frac{2^{t(1-1/m)}}{t}$
\end{enumerate}

\end{thm}

The constant  $\alpha_m$ in item (2) of Theorem \ref{thm:main} is the unique positive root of the polynomial,
$$ z^{m}-z^{m-1}-z^{m-2}-...-z-1.$$ 
The  $\{\alpha_m\}_{m=2}^{\infty}$ are increasing in $m$,
satisfy  $2(1-\frac{1}{2^{m}})\leq \alpha_m  \leq 2$, and go to 2 as $m \rightarrow \infty$. See \cite{Dr} for the details and the proofs of these algebraic properties. Using these properties on the functions in Theorem \ref{thm:main} we have,

\begin{cor}
\label{cor:main} The asymptotic growth rate of the primitive reciprocal geodesics in  the $m$-thick part, $S_m$, converges to the asymptotic growth rate of the primitive reciprocal geodesics on the modular orbifold, as $m \rightarrow \infty$. Similarly,
the asymptotic growth rate of the primitive closed geodesics in $S_m$ converges, up to a multiplicative constant,  to the asymptotic growth rate of the primitive closed geodesics on the modular orbifold, as $m \rightarrow \infty.$
\end{cor}

\begin{rem}  We derive the exact size  of the full set of the  low lying reciprocal geodesics, reciprocal geodesics, and closed geodesics of word length exactly $2t$ (see Table \ref{table: main results}), allowing us to achieve coarse bounds and hence asymptotic growth rates for word length $\leq2t$; thus proving the results of Theorem \ref{thm:main}. 
\end{rem}

\begin{table}
\begin{tabular}{|l|l|l|}

 \hline\xrowht{20pt}
{\bf Geodesic Set} & {\bf Bijection To} & {\bf Cardinality}\\

      \hline\xrowht{25pt}
   $\text{Geodesics of length } 2t $& \vtop{\hbox{\strut Lyndon binary words}\hbox{\strut (primitive and nonprimitive)}\hbox{\strut of length $t$}} & $\frac{1}{t} \displaystyle\sum_{j=1}^{t}  2^{\text{gcd}(j,t)}-2$  \\
   \hline\xrowht{25pt}
   \text{Reciprocal geodesics  of length $4t$} & Compositions of $t$ & $2^{t-1}$  \\
   \hline\xrowht{25pt}
   $\text{Geodesics in $S_m$ of length $2t$}$& \vtop{\hbox{\strut $m$-Lyndon binary words}\hbox{\strut of length $t$ }} & $\geq \frac{2^{t-\frac{t}{m}-1}}{t}$  \\
\hline\xrowht{25pt}
\vtop{\hbox{\strut Reciprocal geodesics in $S_m$}\hbox{\strut of length $4t$}} & \vtop{\hbox{\strut Compositions of $t$ with}\hbox{\strut parts bounded by $m$}} & $\left\lfloor \frac{1}{2}+\frac{\alpha_m -1}{2+(m+1)(\alpha_m -2)}\alpha_m^{t}  \right\rfloor$    \\
\hline

\end{tabular}
\vskip15pt
  \centering 
  \caption{Cardinality of geodesic classes}\label{table: main results}
\end{table}


 The study of asymptotic growth rates of geometric lengths of various classes of closed geodesics has a long and storied history beginning with Huber's result for all closed geodesics, to Mirzakhani's growth rate of the simple closed geodesics,  to more general results for non-simple closed geodesics and reciprocal geodesics \cite{Boca,B-K2,B-K3,Bus,ErPaSo,ErSo,Mirz,Sar}.  Concurrently  there is the study of such geodesics in terms of word length or equivalently  primitive conjugacy classes  and their word length  growth rates leading to  more abstract, algebraic investigations of groups such as surface groups or  free groups \cite{CalLou,Er,GubSap,Park,Ri,Tra}. 
 
 Our focus in this paper is on so called low lying and reciprocal words in the modular group. These are the lifts of the low lying and reciprocal geodesics, resp. Although neither of these sets form a group they have the minimal properties needed to consider the growth rate of their  primitive conjugacy classes. Namely, these subsets of the modular group  are comprised of infinite order elements,  are conjugacy invariant, are closed under taking powers,  and the unique positive power primitive element in the modular group  is also a member of the subset.

  We take,  for the most part,  a combinatorial approach to determine the growth rate of the primitive conjugacy classes.
 Typically  in such arguments  a convenient normal form for the conjugacy classes  is used and then counted.   Of course, one needs to determine when two elements in normal form  represent the same conjugacy class. In the case of reciprocal words we prove a crucial lemma (Lemma \ref{lem: normal form}) showing exactly two elements in normal form are conjugate and identifying these two elements. The fact that there are two primitive conjugates  was first  proven in Sarnak \cite{Sar} using  different  methods.

 This normal form using the isomorphism  from  
 $\mathbb{Z}_2\ast \mathbb{Z}_3$ to $PSL(2,\mathbb{Z})$ allows us to represent a closed geodesic as a product of    parabolic elements.  How deep a geodesic wanders into  the cusp is directly related to the exponents  of these parabolics. See Lemma \ref{lem:wandering deep in cusp} for  a precise statement. 
  
The paper is organized so that in section \ref{sec:basics and notation} we derive some of the elementary but key lemmas  as well as set up notation. In section \ref{countingbinarywords} we talk about the normal form of a reciprocal word and prove a crucial lemma
which identifies when two such words in normal form are conjugate. 
In sections  \ref{sec: counting primitive conjugacy}  and
\ref{sec: counting reciprocal words},
we determine the size of the  conjugacy classes of elements in  $\mathbb{Z}_2\ast \mathbb{Z}_3$  of length $2t$ and  reciprocal words of length  $4t$ as well the primitive classes of these sets. In section \ref{sec:lying low} we identify   
the low lying conjugacy  classes of length $2t$ with so called 
$m$-Lyndon words of length $t$, and derive an effective lower bound for the growth of such words.  We next construct  a bijection between the conjugacy classes of $m$-low lying reciprocal  words and compositions with parts bounded by $m$
allowing us to count these classes.  Section \ref{sec:representations and geodesic excursion} relates the cusp geometry of the modular orbifold with geodesic excursions  into the cusp. Finally in section \ref{sec:all together} we put the work of the previous sections together to prove our main theorem.


\section{Basics and notation} \label{sec:basics and notation}

We use the notation $f \sim g$ to mean asymptotic to and the symbol $f \gtrsim g$ to mean that there exists a constant $C$ and a $t_0$ so that, $f(t) \geq Cg(t)$, for all $t \geq t_0$.

Consider the group $G=\mathbb{Z}_2\ast \mathbb{Z}_3$. Assume the generator of $\mathbb{Z}_{2}$   is  $a$ and the generator of  $\mathbb{Z}_3$ is   $b$.  An element $g \in G$ is {\it primitive} if it is not a non-trivial power of another element of $G$. 
The {\it word length}  of $g$ denoted  $\|g\|$, is  the minimum length among all words representing $g$ using the symmetric set of generators  
$\left\{a,b,b^{-1}\right\}$. Set  $\mathcal{W}= \{\text{reduced words in the generators of $G$}\}$. The conjugacy class of $g \in G$ is denoted, $\left[g \right]$. 
For a positive integer  $s$, since conjugation  commutes with taking powers,  we may define $\left[g \right]^{s}:=\left[g^{s} \right]$. 
The {\it length of a conjugacy class}  $[g]$ is given by
$\|[g]\| = min \left\{\|h\| :h\in [g] \right\}$. A word in 
$\mathcal{W}$ is {\it cyclically reduced} if any cyclic permutation of it is a reduced word. Though cyclically reduced words  in a conjugacy class are not unique they do realize the minimum length in the conjugacy class.  In fact, all conjugates of a cyclically reduced word are cyclic permutations of each other.  For the basics on combinatorial group theory see \cite{LyndSch, Mag}.

We call a reduced word that begins with $a$ and ends with $b$ or $b^{-1}$ an $(ab)$-word.  Similarly, we have $(aa)$-, $(bb)$-, $(ba)$-words. We remark the obvious but important fact that an $(ab)$- or $(ba)$-word is cyclically reduced.

We have the following fundamental lemma.

 \begin{lem}  \label{lem: conjugates}

 Let $x \in \mathcal{W}$ where $x$ is not conjugate to one of the generators, that is, not conjugate to $a$, $b$, or $b^{-1}$. Then 
 \begin{enumerate}
  \item  $x$ is conjugate to an (ab)-word $y$ with 
  $\|x\| \geq \|y\|$.
  \item The only conjugates of the word 
  $ab^{\epsilon_1}\dots ab^{\epsilon_{t}}$  that are   (ab)-words are its even cyclic   permutations. That is, $ab^{\epsilon_{t}} ab^{\epsilon_1}\dots ab^{\epsilon_{t-1}}$
  and so on. 
  \item  If $y$ is an (ab)-word and $x^s=y$ for 
  $s$ a positive integer, then $x$ is an (ab)-word  and $s\|x\|=\|y\|$.  
  \item   If  $\left[x\right]^s=\left[y\right]$ then 
  $s \|[x]\|=\|[y]\|.$

\end{enumerate}
\end{lem}

\begin{proof}  Items (1) - (3) follow immediately.  To prove  item (4)  we may assume, by conjugating if necessary, that $y$ is an $(ab)$-word. 
Now, by assumption there exists an $x$ so that 
 $x^s=y$, and hence $x$ is an $(ab)$-word by item (3).  Moreover we have  $ \|[y]\|=\|y\|=s\|x\|=s\|[x]\|$, where the second and third equalities also follow from item (3).

\end{proof}

\begin{rem}  In the group $G=\mathbb{Z}_2\ast \mathbb{Z}_3$ the infinite order elements are  the positive  power of a unique, primitive element. Although this can proven using purely combinatorial methods, the easiest way to see this is by  the fact that $G$ has a discrete, faithful representation into $PSL(2,\mathbb{Z})$. 

 \end{rem}
 
  Throughout this work we consider  subsets of $G$  that have the  following minimal properties.
   
 \begin{definition}
 A set $\mathcal{A} \subset G$ made up of infinite order elements  is said to satisfy condition (*) if the following properties are satisfied.

\begin{itemize}

\item The positive or negative power of any element in 
$\mathcal{A}$ is also in $\mathcal{A}$.
  \item For any element in $\mathcal{A}$, the unique primitive element in $G$ for which it is a  positive power  is also in $\mathcal{A}$. 
  
   \item $\mathcal{A}$ is conjugacy invariant. 
  
\end{itemize} .

 \end{definition}

In the sequel the subsets $\mathcal{A}$ will denote either   infinite order words  in $G$,  reciprocal words (to be defined later), or low lying words (to be defined later). For now we proceed abstractly with the set $\mathcal{A}$ satisfying condition (*),  we fix notation, and derive some basic facts.

Setting  $\mathcal{A}^{p}=\{\text{primitive elements of } 
\mathcal{A}\}$  we have,
$\mathcal{A}^{p} \subseteq  \mathcal{A} \subseteq
\mathcal{W}$. Since each of these subsets is closed under conjugation by elements of $G$ we define the conjugacy classes of these subsets by capitals: ${A}^{p}$, $A$, and $W$, respectively. Note that $W$ is the full set of conjugacy classes in $G$. We denote the {\it nonprimitive conjugacy classes}  in $A$ by 
$A^{np}$.  

  For various choices of $\mathcal{A}$  we are interested in the growth rate of primitive conjugacy classes in $\mathcal{A}$.
    Here growth is measured by word length in terms of the generators.

For $t$ a positive integer we use $t$ as a subscript  to denote the elements in that set of length $t$.  Similarly, we use
$\leq t$  as a subscript to denote the elements in the set of length $\leq t$.
For example, $A_t$ denotes the conjugacy classes in $A$ of length $t$, and $A_{\leq t}$  denotes the conjugacy classes in $A$ of length less than or equal to $t$.  The growth function for the set $A$ is denoted by $|A_{\leq t}|$.  A {\it proper divisor} of $t$ is a positive integer that divides $t$ but is not $1$ or $t$. 

We next define a map from  primitive conjugacy classes to non-primitive conjugacy classes given by a power map.

\begin{lem} \label{lem: bijection}
 The map $\iota : \bigcup_{s | t}  {A}^{p}_s 
 \rightarrow {A_{t}^{np}}$ given by $[x] \mapsto [x^{t/s}]$ is well-defined and a bijection. That is, the nonprimitive conjugacy classes in 
$A_{t}$ are in one-to-one correspondence with elements of  $\bigcupdot_{s | t} {A}^{p}_s$, where  the union is over all  proper divisors, $s$,  of $t$.

\end{lem}

\begin{proof} $\iota$ is well-defined since powers commute with conjugation. To prove surjectivity, suppose $[y] \in A^{np}_t$ and hence there exists $[x] \in A^{p}_s$ so that $[x]^{n}=[y]$, where $n$ is a positive integer greater than 1. By item (4) of Lemma \ref{lem: conjugates},  $n \|[x]\|=\|[y]\|$ and hence $s$ divides $t$.  If $s=1$, then $n=t$ and $\|[x]\|=1$.  So $x$ is conjugate to $a$ or $b^{\pm1}$, which implies $x$ has finite order.  Thus $s$ properly divides $t$.

Injectivity follows from establishing the following two items, which we leave to the reader.  

1) if $[x_1] \neq [x_2]$
in $A^{p}_s$, then $\iota ([x_1])$  is not conjugate to 
$\iota ([x_2])$.

2) if $s_1$  and $s_2$ divide $t$ , $s_1 \neq s_2$,
then $\iota (A^{p}_{s_1}) \cap  \iota (A^{p}_{s_2})=\
 \emptyset$.

\end{proof}

 We have established,

\begin{prop}\label{prop:primitives sum}
Suppose $\mathcal{A}\subset G$ satisfies condition (*).  Then
\begin{equation*}
A^{p}_{t}=A_{t}- \bigcupdot_{s | t} \iota(A^{p}_{s}) \text{  and  } |A^{p}_{t}|=|A_{t}|- \sum_{s | t} |A^{p}_{s}|
\end{equation*} 
where the union and sum  are  over all proper divisors, $s$, of  $t$. \end{prop}




 



Our goal in the next few sections is to compute the asymptotics as $t\to\infty$ of the function $|A^p_{2t}|$  and $|A^p_{\leq 2t}|$ for various choices of $\mathcal{A}$.

Set  $\mathcal{R}= \{\textrm{$xy$ : $x, y$ are distinct order 2 elements in $G$}\}$. Denoting the commutator of $x$ and $y$ by $\left[x,y\right]$, and noting that there is one conjugacy class of order 2 elements, we may write
\begin{eqnarray*}
\mathcal{R}&=& \left\{[xax^{-1},xyx^{-1}] : x,y \in G, y \neq a \right\}\\
&=& \left\{x[a,y]x^{-1} : x,y \in G, y \neq a \right\}.
\end{eqnarray*}

We call the elements of $\mathcal{R}$ {\it reciprocal words}.  We remark that   $\mathcal{R}$ is closed under taking powers and moreover   if an element   of 
$\mathcal{R}$ is a power of the unique, primitive $y \in G$, then 
$y$ is also in  $\mathcal{R}$. Thus $\mathcal{R}$ satisfies condition $(*)$. Denote the primitive conjugacy classes in $R$ by $R^{p}$. The conjugacy classes of $\mathcal{R}$ correspond exactly to the set of  reciprocal geodesics on the modular orbifold.  This is because a lift of a reciprocal geodesic is the product of a pair of order two elements in $PSL(2,\mathbb{Z})$.

\begin{rem}
Any reciprocal word  conjugated to an (ab)-word has the form
$$w=[a,\gamma]=ab^{\epsilon_1}...ab^{\epsilon_{t}}ab^{-\epsilon_{t}}...ab^{-\epsilon_1}$$ where $\epsilon_i=\pm1$ and $\gamma$ is a (bb)-word.  With this in mind we define the normal form of a reciprocal word to be
$[a,\gamma]$ where  $\gamma$ is   $\textrm{a (bb)-word}$ and their full set to be denoted,

\begin{equation*}\mathcal{N}=\{[a,\gamma] : \gamma\,\, \textrm{a (bb)-word}\}.
\end{equation*} 

Note that nonprimitive elements of $\mathcal{N}$ are powers of elements of the same form.  That is, $[a,\gamma]=[a,\beta]^n$, where $[a,\beta]$ is primitive.
\end{rem}

  
 

   




\section{Binary Words and the Normal Form for a Reciprocal Word}
\label{countingbinarywords}
Our interest is in counting words in $G$. To make our computations less cumbersome we identify $(ab)$-words in $G$ with binary words.  Namely, we identify the $(ab)$-word,\newline $ab^{\epsilon_0}ab^{\epsilon_1}\dots ab^{\epsilon_{t-1}}$ with the binary word, $(\epsilon_0,\epsilon_1,\dots,\epsilon_{t-1})$, where $\epsilon_i=\pm1$.  Denote the set of all binary words of length $t$ by $X_t$.

Focusing on reciprocal words, the length of a reciprocal word in normal form,  $\mathcal{N}$, is a multiple of 4.   Hence we identify $\mathcal{N}$ with the subset  

\begin{equation*}
Y_{2t}=\{(\epsilon_0,\dots,\epsilon_{2t-1}): \epsilon_j=-\epsilon_{2t-j-1},\, \textrm{for all} \,j=0,\dots,2t-1\}\subset X_{2t}.
\end{equation*} 

 Conjugate words of an $(ab)$-word in the same $(ab)$-form are cyclic permutations of even order.   With this in mind, using the bijection, we denote this cyclic action on $X_{2t}$ by 
 $\alpha^k(\epsilon_0,\dots,\epsilon_{2t-1})=(\epsilon'_0,\dots,\epsilon'_{2t-1})$, where $\epsilon'_j=\epsilon_{j-k}$ for all $j=0,\dots,2t-1$.  Here we use the convention that the subscripts are modulo $2t$.

\begin{lem}
\label{lem:alpha^k(x)}
Fix $k$. For any  $(\epsilon_0,\dots,\epsilon_{2t-1})\in X_{2t}$ we have,
 
\begin{itemize}

\item   $\alpha^{k}(\epsilon_0,\dots,\epsilon_{2t-1})=(\epsilon_0,\dots,\epsilon_{2t-1})$ if and only if 

\begin{equation}
\label{eq:alpha^k(x)=x}
\epsilon_j=\epsilon_{j-k},\,  \text{ for all } j.
\end{equation}

\item  $\alpha^k(\epsilon_0,\dots,\epsilon_{2t-1})\in{Y}_{2t}$ if and only if 

\begin{equation}
\label{eq:scriptNcond}
\epsilon_{j-k}=-\epsilon_{2t-j-1-k},\, \text{ for all } j.
\end{equation}

\end{itemize}
\end{lem}

\begin{proof} Item (1) follows from the definition of the cyclic action. 
For item (2), set 

\[(\epsilon'_0,\dots,\epsilon'_{2t-1})=\alpha^k(\epsilon_0,\dots,\epsilon_{2t-1}).\]  

$(\epsilon'_0,\dots,\epsilon'_{2t-1})\in Y_{2t}$ if and only if $\epsilon'_j=-\epsilon'_{2t-j-1}$, but using the definition of the cyclic action this is equivalent to $\epsilon_{j-k}=-\epsilon_{2t-j-1-k}$. 

\end{proof}

\begin{lem}
\label{lem:scriptN}
If $(\epsilon_0,\dots,\epsilon_{2t-1})\in{Y_{2t}}$, then $\alpha^t(\epsilon_0,\dots,\epsilon_{2t-1})\in Y_{2t}$.
\end{lem}

\begin{proof}
We need to show that (\ref{eq:scriptNcond}) holds for the case $k=t$.  That is, $\epsilon_{j-t}=-\epsilon_{t-j-1}$.  Now since $(\epsilon_0,\dots,\epsilon_{2t-1})\in Y_{2t}$ we have that

\begin{eqnarray*}
\epsilon_{j-t} &=& -\epsilon_{2t-(j-t)-1}\\
&=&-\epsilon_{3t-j-1}\\
&=&-\epsilon_{t-j-1}.
\end{eqnarray*}

\end{proof}

\begin{lem}
\label{lem:powersinN}
Fix $k$ and $(\epsilon_0,\dots,\epsilon_{2t-1})\in Y_{2t}.$ Then $\alpha^k(\epsilon_0,\dots,\epsilon_{2t-1})\in Y_{2t}$ if and only if

\begin{equation}
\label{eq:iffTuple}
\epsilon_j=\epsilon_{j-2k}, \,\textrm{for all $j$}.
\end{equation}

 \end{lem}

\begin{proof}
The condition that $(\epsilon_0,\dots,\epsilon_{2t-1})\in Y_{2t}$ is $\epsilon_j=-\epsilon_{2t-j-1}$ for all $j$, while the condition $\alpha^k(\epsilon_0,\dots,\epsilon_{2t-1})\in Y_{2t}$ is equation (\ref{eq:scriptNcond}).  Replace $j$ with $2t-j-1+k$ in equation (\ref{eq:scriptNcond}).  We get:

\begin{eqnarray}
\epsilon_{2t-j-1}&=&-\epsilon_{2t-(2t-j-1+k)-1-k}\label{eq:3.3}\\
&=&-\epsilon_{j-2k}.
\end{eqnarray}
 
Using the condition to be in $Y_{2t}$ we replace the left-hand side of equation (\ref{eq:3.3}) with $-\epsilon_j$ and we have established equation (\ref{eq:iffTuple}).

For the reverse direction, replace $\epsilon_j$ with $\epsilon_{j-2k}$ in the condition to be an element in $Y_{2t}$.  That is, $\epsilon_{j-2k}=-\epsilon_{2t-j-1}$.  Replacing $j$ with $j+k$ in this equation yields the desired result.

\end{proof}

With an eye toward applications to words in the group $G$ we define a notion of primitivity for a binary word.

\begin{definition} For any positive integer $t$, an element $x\in X_{t}$ is called nonprimitive if there exists $s$ which properly divides $t$ so that $x$ is the juxtaposition of $s$-subtuples where each subtuple is the same. That is, there exists $z\in X_{t/s}$ so that $x=z^s$. Otherwise, we say $x$ is primitive.  
\end{definition}

Every element in $X_t$ is the positive integer power of some primitive subtuple.  Moreover, if $x\in Y_{2t}\subset X_{2t}$ is nonprimitive then $x=y^{s}$, where $y\in Y_{2t/s}$.

Let $\mathcal{O}$ be the orbit in $X_{2t}$ of an element in $Y_{2t} \subset X_{2t}$ under the action of $\alpha$.  We are interested in $\mathcal{O}\cap Y_{2t}$.  

We set $k_0$ to be the {\it smallest power of $\alpha$} for which an element of $\mathcal{O}\cap Y_{2t}$ maps back to $\mathcal{O}\cap Y_{2t}$. Note that $k_0$ is an invariant of the orbit and $1\leq k_0\leq t$.

\begin{lem} \label{lem: orbits of x} We have

\begin{enumerate}
\item
For any integer $n$, $\alpha^{nk_0}(\mathcal{O}\cap Y_{2t})=\mathcal{O}\cap Y_{2t}$.
\item
$\alpha^l(\mathcal{O}\cap Y_{2t})\cap(\mathcal{O}\cap Y_{2t})=\emptyset$, for $l$ not a multiple of $k_0$.
\item
$\alpha^{k_0}(x)\neq x$, for any $x\in\mathcal{O}\cap Y_{2t}$.
\end{enumerate}
\end{lem}

\begin{proof}  Set $x=(\epsilon_0,\dots,\epsilon_{2t-1})\in\mathcal{O}\cap Y_{2t}$ throughout this proof.

To prove item (1), it's enough to show that for any integer $n$, $\alpha^{nk_0}(\mathcal{O}\cap Y_{2t})\subset\mathcal{O}\cap Y_{2t}$.  Noting that $\alpha^{k_0}(x)\in\mathcal{O}\cap Y_{2t}$, we have by Lemma \ref{lem:powersinN}, $\epsilon_j=\epsilon_{j-2k_0}=\dots=\epsilon_{j-2nk_0}$, for all $j$. Therefore, again by Lemma \ref{lem:powersinN}, $\alpha^{nk_0}(x)\in\mathcal{O}\cap Y_{2t}$.

For item (2), suppose for contradiction $\alpha^l(x)\in\mathcal{O}\cap Y_{2t}$ and $nk_0<l<(n+1)k_0$.  Then $\alpha^{l-nk_0}(\alpha^{nk_0}(x))=\alpha^l(x)\in\mathcal{O}\cap Y_{2t}$. On the other hand by item (1), $\alpha^{nk_0}(x)\in\mathcal{O}\cap Y_{2t}$ and so $\alpha^{l-nk_0}(\mathcal{O}\cap Y_{2t})\cap(\mathcal{O}\cap Y_{2t})\neq\emptyset$, contradicting our assumption that $k_0$ is minimal.

To prove item (3), assume $\alpha^{k_0}(x)=x$ and hence by equation (\ref{eq:alpha^k(x)=x}), $\epsilon_j=\epsilon_{j-k_0}$, for all $j$, and there is a repeating subtuple $y$  of length $k_0$ which fills out 
$x$.  Moreover, this subtuple must lie in $Y_{k_0}$ and since the elements of $Y_{k_0}$ have   even length, $k_0$ must be even. 
 On the other hand, setting $k=\frac{k_0}{2}$ we have $\epsilon_{j-2k}=\epsilon_{j-k_0}=\epsilon_{j}$, for all $j$, where the last equality follows from equation (\ref{eq:alpha^k(x)=x}).  By Lemma \ref{lem:powersinN} this contradicts  the minimality of $k_0$ and  thus $\alpha^{k_0}(x)\neq x$.

\end{proof}

\begin{prop}
\label{prop:2distinctelts}
Assume  $x \in \mathcal{O}\cap Y_{2t}$.
$\alpha^{nk_0}(x)=x$ for $n$ even, and $\alpha^{nk_0}(x)=\alpha^{k_{0}}(x)$ for $n$ odd. Namely, $\mathcal{O}\cap Y_{2t}$ consists of 2 distinct elements, 
$\{x,\alpha^{k_0}(x)\}$.
\end{prop}

\begin{proof} Set $x=(\epsilon_0,\dots,\epsilon_{2t-1})\in\mathcal{O}\cap Y_{2t}$.
We note, by items (1) and (2) of Lemma \ref{lem: orbits of x}, that  all the points in $\mathcal{O}\cap Y_{2t}$ are of the form $\alpha^{nk_0}(x)$. Since $\alpha^{k_0}(x)\in\mathcal{O}\cap Y_{2t}$ and assuming $n$ is even, we have 
$\epsilon_j=\epsilon_{j-2k_{0}}=\epsilon_{j-4k_{0}}=\dots=\epsilon_{j-nk_{0}}$ by Lemma \ref{lem:powersinN}. Hence by Lemma \ref{lem:alpha^k(x)}, $\alpha^{nk_0}(x)=x$.  For odd $n=2m+1$, we have that $\alpha^{nk_0}(x)=\alpha^{(2m+1)k_0}(x)=\alpha^{k_0}(\alpha^{2mk_0}(x))=\alpha^{k_0}(x)$.  However, $\alpha^{k_0}(x)\neq x$ by item (3) of Lemma \ref{lem: orbits of x}.

\end{proof}

\begin{rem}A schematic picture emerges.  We picture  $Y_{2t}$ as the diagonal in $X_{2t}$ and the $\langle \alpha\rangle$-orbit of a point in $Y_{2t}$ as intersecting $Y_{2t}$ in exactly 2 distinct points, and the number of  all orbit  points in $X_{2t}$ being $2k_{0}$.
\end{rem}

\begin{prop}
Let $x \in Y_{2t}$. Then $x$ is primitive if and only if $k_0=t$.
\end{prop}

\begin{proof}
We first remark that $x$ is nonprimitive if and only if there exists a minimal subtuple $y\in Y_{2s}$ repeated $\frac{t}{s}$-times, where $s$ properly divides $t$, giving $x$.  We have $k_0<2s\leq t$, where the left inequality follows from item (3) of Lemma \ref{lem: orbits of x}.  Thus $k_0<t$. On the other hand, if $x$ is primitive then there is no proper $s$ and thus $k_0=t$.
\end{proof}

We now return to the main objective of this section to consider the normal form of reciprocal words.  We remind the reader of the bijection
\[\mathcal{N}_{4t}=\{[a,\gamma] : \gamma\,\, \textrm{a (bb)-word of length}\,\, 2t-1\}\rightarrow Y_{2t}\]

given by

\[ab^{\epsilon_0}...ab^{\epsilon_{t-1}}ab^{-\epsilon_{t-1}}...ab^{-\epsilon_0}\mapsto (\epsilon_0,\dots,\epsilon_{t-1},-\epsilon_{t-1},\dots,-\epsilon_0).\]
Using the bijection with Proposition \ref{prop:2distinctelts}, and noting  that when  $[a,\beta]$ is primitive,  $\beta$ cannot be   of order two,  we have proven,
 
\begin{lem} \label{lem: normal form}
 Each conjugacy class of an element of $\mathcal{R}$  has exactly two  representatives  in the normal form $\mathcal{N}$. Namely, the two conjugates in $\mathcal{N}$ are $[a,\beta]^n$ and $[a,\beta^{-1}]^n$,  where $[a,\beta]$ is primitive, $n$ is a unique positive integer, and $\beta$ is a unique (bb)-word not of order 2.
 \end{lem} 
  
 \begin{rem}In proving Lemma \ref{lem: normal form} we largely took a combinatorial point of view.  However, using different methods in \cite{Sar}, this lemma is  proven for primitive reciprocals  in  $PSL(2, \mathbb{Z})$.
 \end{rem}
  \vspace{.2in}


\section{Counting  conjugacy classes in   $\mathbb{Z}_2\ast \mathbb{Z}_3$} \label{sec: counting primitive conjugacy}

The goal of this section is to investigate the growth rate of primitive conjugacy classes in $G$.  We first  recall the Burnside lemma. Let $G$ be a finite group acting on a set $B$. Then,
\begin{equation*}
|B/G|= \frac{1}{|G|} \sum_{g \in G} |Fix (g)|
\end{equation*}
where $Fix(g)$ is the fixed point set of $g$.

Once again we appeal to binary words for our computations.  In this section we work with binary words of length $t$.  Fix a positive integer $t$ and recall  $X_t=\{(\epsilon_0,...,\epsilon_{t-1}): \epsilon_i=\pm1\}$.  We consider the cyclic permutation map, $\alpha :X_t \rightarrow X_t$, given by 
$(\epsilon_0,...,\epsilon_{t-1}) \mapsto  (\epsilon_{t-1},\epsilon_0....,\epsilon_{t-2}).$ Thus the group generated by $\alpha$ is cyclic of order $t$. We will compute the size of the fixed point sets $Fix(\alpha^{j})$ for $j=1,..., t$. Of course $|Fix(\alpha^t)|=|Fix(id)|=2^t$.
If $j$ is relatively prime to $t$, then  
$|Fix(\alpha^{j})|=2$.  More generally, if $(\epsilon_0,...,\epsilon_{t-1})$ is fixed by $\alpha^j$, then for each $i=0,1,...,t-1$, $\epsilon_i=\epsilon_{i+j}$, where the subscripts are modulo $t$. Since $\epsilon_i$ has two possibilities $(\pm1)$ and since the number of 
equivalence classes of $\epsilon_i$ (subscript modulo $t$) is  equal  to the $\text{gcd}(j,t)$, we have 
$|Fix(\alpha^{j})|=2^{\text{gcd}(j,t)}$.  Thus the number of points in the quotient is
\begin{equation}
\label{eq:BurnsideBinary}
|X_t/<\alpha>|=
\frac{1}{t} \sum_{j=1}^{t}  2^{\text{gcd}(j,t)}.
\end{equation}

Denote the set of all words in $G$ of $(ab)$-form with length $2t$ by

\begin{equation*}
\mathcal{W}_{2t} (ab)=
\big\{ab^{\epsilon_0}...ab^{\epsilon_{t-1}}:   
\epsilon_i =\pm 1\}.
\end{equation*}

Consider the action 
$\beta :\mathcal{W}_{2t} (ab) \rightarrow  \mathcal{W}_{2t} (ab)$
given by  $ab^{\epsilon_0}...ab^{\epsilon_{t-1}} \mapsto 
ab^{\epsilon_{t-1}}ab^{\epsilon_0}...ab^{\epsilon_{t-2}}$.

 The group  $<\beta >$ is  cyclic of order $t$.

\begin{lem}  \label{lem:orbit points}
\begin{equation*}
|\mathcal{W}_{2t} (ab)/<\beta>|=
\frac{1}{t} \sum_{j=1}^{t}  2^{\text{gcd}(j,t)}
\end{equation*}

\end{lem}

\begin{proof}

Note that 
\begin{equation*}
(\epsilon_0,...,\epsilon_{t-1}) \mapsto ab^{\epsilon_0}...ab^{\epsilon_{t-1}}
\end{equation*}
 is an equivariant bijection  between the 
$<\alpha>$ action on   $X_t$  and the $<\beta>$ action on  
$\mathcal{W}_{2t} (ab)$. The result follows using (\ref{eq:BurnsideBinary}).
 
\end{proof}

\begin{thm} \label{thm:Wasymptotics}
\leavevmode
 
\begin{enumerate}
\item
$|W_{2t}|=\frac{1}{t} \displaystyle\sum_{j=1}^{t}  2^{\text{gcd}(j,t)}$ \\
\item
$|W_{2t}|\sim\dfrac{2^t}{t}$, as $t\to\infty$\\
\item
$|W_{\leq 2t}| \sim \dfrac{2^{t+1}}{t}$,
as $t\to \infty$
\end{enumerate}

\end{thm}

\begin{proof}
A proof of item (1) appears in \cite{Tra}, however for completeness we supply a proof. Suppose  $[w] \in W_{2t}$. Then in the conjugacy class of $w$  there is a representative  in  $\mathcal{W}_{2t} (ab)$.  Now the only other conjugates in  $\mathcal{W}_{2t} (ab)$ are the ones equivalent under the action of  $<\beta>$.  There is a one-to-one 
correspondence between the set of conjugacy classes $W_{2t}$ and
$\mathcal{W}_{2t} (ab)/<\beta>$, and hence by Lemma  \ref{lem:orbit points} the result follows.

To prove item (2) we begin by noting that gcd$(t,t)=t$ and gcd$(j,t)\leq \frac{t}{2}$ for $j<t$.  It follows that

\begin{equation*}
\frac{2^t}{t}\leq|W_{2t}|= \frac{1}{t}\left(\sum_{j=1}^{t-1}2^{gcd(j,t)}+2^t\right)\leq\left(\frac{t-1}{t}\right)2^{t/2}+\frac{2^t}{t}.
\end{equation*}

Since
\begin{equation*}
\frac{\left(\frac{t-1}{t}\right)2^{t/2}}{\frac{2^t}{t}}\to0
\end{equation*}
the claimed asymptotic follows.

For item (3) note that

\begin{equation*}\label{eq: W}
|W_{\leq 2t}|=3+\sum_{n=1}^{t} |W_{2n}|=3+
\sum_{n=1}^{t} 
\frac{1}{n} \sum_{j=1}^{n}2^{\text{gcd}(j,n)}.
\end{equation*}
We remark that the term $3$ appears above since there are 3 length one conjugacy classes. Using the same reasoning from item (2) we have,

\begin{equation*} \label{ }
\sum_{n=1}^{t} 
\frac{2^{n}}{n} \leq  |W_{\leq 2t}| 
\leq 3+  \sum_{n=1}^{t}  \frac{2^{n}}{n}+ 
\sum_{n=1}^{t}  2^{n/2}.
\end{equation*}  

Since
\begin{equation*}\label{ }
 \frac{\displaystyle\sum_{n=1}^{t}  2^{n/2}}
{\displaystyle\sum_{n=1}^{t}  \frac{2^{n}}{n}}
\rightarrow 0
\end{equation*}
we have that, 
$|W_{\leq 2t}| \sim \displaystyle\sum_{n=1}^{t}\frac{2^n}{n}$.
Finally, an application of the Stolz-Cesaro Theorem \cite{Mur} yields $\displaystyle\sum_{n=1}^{t}\dfrac{2^n}{n}\sim \dfrac{2^{t+1}}{t+1}$.\end{proof}

\begin{lem} \label{lem: {W}_{2t}^{np}}
\leavevmode
\begin{enumerate}
\item
$|W_{2t}^{np}|\leq\frac{1}{2}t2^{t/2}$\\
\item
$|W_{\leq2t}^{np}|\leq\frac{1}{2}t^22^{t/2}$
\end{enumerate}
\end{lem}

\begin{proof}
For item (1), using Proposition \ref{prop:primitives sum} we have

\begin{equation*}
|W_{2t}^{np}|=\sum_{s|t} |W_{2s}^{p}|\leq\sum_{s|t} |W_{2s}|=\sum_{s|t}\frac{1}{s}\sum_{j=1}^s2^{gcd(j,s)}\leq\sum_{s|t}\frac{1}{s}\sum_{j=1}^s2^s=\sum_{s|t}2^s\leq\frac{t}{2}2^{t/2}
\end{equation*}

where the last inequality follows from the fact that the largest proper divisor of $t$ is $\frac{t}{2}$ and there are at most $\frac{t}{2}$ divisors. For item (2) we apply item (1):

\begin{equation*}
|W_{\leq2t}^{np}|=\sum_{n=1}^{t}|W_{2n}^{np}|\leq\sum_{n=1}^{t}\frac{1}{2}n2^{n/2}\leq\frac{1}{2}t2^{t/2}\sum_{n=1}^t1=\frac{1}{2}t^22^{t/2}.
\end{equation*}

\end{proof}

\begin{thm}\label{thm:Wasymptotics primitive}
\leavevmode
\begin{enumerate}
\item
$|W_{2t}^p|\sim|W_{2t}|\sim\dfrac{2^t}{t}$, as $t\to\infty$\\
\item
$|W_{\leq 2t}^p|\sim|W_{\leq2t}|\sim\dfrac{2^{t+1}}{t}$, as $t\to\infty$
\end{enumerate}

\end{thm}

\begin{proof}
Applying Lemma \ref{lem: {W}_{2t}^{np}}, we have

\begin{equation*}
|W_{2t}|-\frac{1}{2}t2^{t/2}\leq|W_{2t}^p|\leq|W_{2t}|.
\end{equation*}

Dividing by $|W_{2t}|$ and noting that $|W_{2t}|\geq\frac{2^t}{t}$, we get

\begin{equation*}
1-\frac{\frac{1}{2}t2^{t/2}}{\frac{2^t}{t}}\leq\frac{|W_{2t}^p|}{|W_{2t}|}\leq 1,
\end{equation*}
which yields the desired asymptotic.

For part (2), we use an analogous argument. We have from Lemma \ref{lem: {W}_{2t}^{np}},

\begin{equation*}
|W_{\leq2t}|-\frac{1}{2}t^22^{t/2}\leq|W_{\leq2t}^p|\leq|W_{\leq2t}|.
\end{equation*}

Dividing by $|W_{\leq 2t}|$ and using that $|W_{\leq2t}|\geq\sum_{n=1}^t\frac{2^n}{n}\geq\frac{2^t}{t}$, we have
\begin{equation*}
1-\frac{\frac{1}{2}t^22^{t/2}}{\frac{2^t}{t}}\leq\frac{|W_{\leq2t}^p|}{|W_{\leq2t}|}\leq 1,
\end{equation*}

giving the desired asymptotic.

\end{proof}


\section{Counting conjugacy classes of reciprocal words}
\label{sec: counting reciprocal words}

In this section we  compute the growth rate of conjugacy classes of reciprocal words in $G=\mathbb{Z}_2\ast \mathbb{Z}_3$.  The conjugacy classes of reciprocal words have length a multiple of $4$. For this reason we use the parameter $4t$ for  ease of  computation.   We have,

\begin{lem}\label{lem:reciprocal conj sizes}
\leavevmode  
\begin{enumerate}
\item
$|R_{4t}|=2^{t-1}$
\item
$|R_{\leq 4t}|=2^t-1$
\end{enumerate}
\end{lem}

\begin{proof}
 Let $[w] \in R_{4t}$. We pick as representative a cyclically reduced word of the form 
$$w=ab^{\epsilon_0}...ab^{\epsilon_{t-1}}ab^{-\epsilon_{t-1}}...ab^{-\epsilon_0}$$ where   
$\epsilon_i =\pm 1$. The result follows since there are exactly two conjugates of this form (Lemma \ref{lem: normal form}) and there are $2^{t}$ words of this form.

For the second claim,

\begin{equation*}
|R_{\leq 4t}|=\sum_{n=1}^t|R_{4n}|=\frac{1}{2}\sum_{n=1}^t2^n=\frac{1}{2}\left[2\left(\frac{1-2^t}{1-2}\right)\right]=2^t-1.
\end{equation*} 
\end{proof}

\begin{lem} \label{lem: {R}_{4t}^{np}}
\leavevmode
\begin{enumerate}
\item
$|{R}_{4t}^{np}| \leq \frac{1}{4} t2^{t/2}$ \\
\item
$|{R}_{\leq 4t}^{np}| \leq \frac{1}{4} t^2 2^{t/2}$
\end{enumerate}
\end{lem}

The proof of the lemma follows in an analagous way to the proof of Lemma \ref{lem: {W}_{2t}^{np}}.  We leave the details to the reader.

\begin{thm}\label{thm:growth prim conjugacy rec words 2}
 \leavevmode
\begin{enumerate}
  \item  $ |R^{p}_{4t}| \sim |R_{4t}| = 2^{t-1}, \text{  as  } t \rightarrow \infty$\\

 \item $ |R^{p}_{\leq 4t}| \sim |R_{\leq 4t}| = 2^t-1, \text{  as  } t \rightarrow \infty$
\end{enumerate}

 \end{thm}

\begin{proof} For part (1), we apply Lemma \ref{lem: {R}_{4t}^{np}} to get

\[|{R}_{4t}| - \frac{1}{4} t2^{t/2} \leq |{R}_{4t}^{p}| \leq |{R}_{4t}|.\]

 Dividing by $|R_{4t}|=\frac{1}{2}2^t$ and letting $t\to\infty$ yields the claimed asymptotic.  For part (2), we again apply Lemma \ref{lem: {R}_{4t}^{np}} to get

\[|{R}_{\leq 4t}| - \frac{1}{4}t^2 2^{t/2} \leq |{R}_{\leq 4t}^{p}| \leq |{R}_{\leq 4t}|.\]

Dividing by $|R_{\leq 4t}|=2^t-1$ and letting $t\to\infty$ yields the claimed asymptotic.

\end{proof}


\section{Lying Low} \label{sec:lying low}

In this section we would like to count the low lying geodesics.  In other words, we consider the growth rate of conjugacy classes of low lying words as well as low lying reciprocal words. For  a positive integer $m$  we say that a word in $\mathcal{W}$ is an {\it m-low lying word}  if when conjugated to an $(ab)$-word, 
$ab^{\epsilon_{1}} ab^{\epsilon_{2}}\dots ab^{\epsilon_{t}}$,  $\epsilon_{i}=\pm 1$,  no  $(m+1)$   consecutive $\epsilon_{i}$ considered cyclically have the same sign. Put another way, the highest exponent  of $ab$ or $ab^{-1}$ considered cyclically in an  $m$-low lying word is at most
$m$.  When the $m$ is understood we simply say that the word is {\it low lying}. 
We denote the conjugacy classes  of  $m$-low lying words of  word length $2t$ as, $L_{2t,m}$, and the $m$-low lying primitive conjugacy classes of words as 
$L^{p}_{2t,m}$.
We note that the property of being $m$-low lying is preserved under conjugation, taking powers, and taking roots, and hence satisfies  condition  $(*)$.

\subsection{Low lying words}

We  fix $m\geq2$  a positive integer. We now consider the growth rate  of the conjugacy classes of all $m$-low lying words. 
Let $\mathcal{L}_{2t,m}(ab)$ be the set of normalized $m$-low lying words of length $2t$. That is, a word of the form,
$w=ab^{\epsilon_{1}} ab^{\epsilon_{2}}\dots ab^{\epsilon_{t}}$, where $\epsilon_{i}=\pm 1$,  and no  
$(m+1)$   consecutive $\epsilon_{i}$ considered cyclically have the same sign.

As before we identify such a word $w$ with the $t$-tuple of $\pm 1's$,  $(\epsilon_1,...,\epsilon_t)$, and of course via this identification we have the notion of a primitive and nonprimitive  $t$-tuple.  The cyclic action on the word $w$ induces  a cyclic action on this $t$-tuple.  With this in mind,  we consider the $t$-tuple of  $\pm 1's$ on a  circle oriented counterclockwise. Within the cyclic equivalence class we identify a distinguished element.  Namely, using the lexicographical ordering (-1 preceeds 1) among cyclic permutations of   $(\epsilon_1,...,\epsilon_t)$  choose the smallest and  call such a word an  {\it $m$-Lyndon binary word}.  An $m$-Lyndon binary word is primitive if and only if none of its  non-trivial cyclic permutations  are equal to it. 
Note that an $m$-Lyndon binary word selects a representative in the conjugacy class of an  $m$-low lying word.  For example, the Lyndon binary word  
  $(-1,-1,-1,1, -1,-1,-1,1, -1,-1,1,1)$ selects the  low lying word 
  $[w]=[(ab^{-1})^{3}(ab)(ab^{-1})^{3}(ab)(ab^{-1})^{2}(ab)^{2}]
  \in L_{24}$.  We have established the following.

\begin{prop} Fix  $m \geq 2$ a positive integer.  $L_{2t,m}$ is bijectively equivalent to the $m$-Lyndon binary words of length $t$.
\end{prop}

\begin{rem}
 In the case that $m\geq t$ and we  restrict to primitive  
binary words,  such words are known  as  {\it Lyndon words}  in the literature.  
\end{rem}

Our next goal is to derive an effective lower bound on the 
number of $m$-Lyndon words  or equivalently the $m$-low lying conjugacy classes of length $2t$. We consider the normal form  of the $m$-low lying words,
$ab^{\epsilon_1}ab^{\epsilon_2}...ab^{\epsilon_t}$ or equivalently $(\epsilon_1,...,\epsilon_t)$ with  cyclic runs of length at most $m$. In order to achieve a lower bound we construct a subset  of $m$-low lying words of length $2t$ or equivalently  $m$-low lying  $t$-tuples. To that end,
we picture $t$ ordered slots and we group them into the first $m$, second $m$, and so on. There are exactly 
$\lceil \frac{t}{m}\rceil$ 
groups where the last grouping has less than or equal to $m$ slots. We color all  these slots   black except the following which are colored red:  the first one 
in the second group of $m$ slots, the first one in the third group of $m$ slots, and so on. If the $t^{th}$ slot (that is, the last slot) is not red it should also be colored red. Let $\mathcal{B}_{2t,m}$ be the subset of $m$-low lying words where any  black slot can be +1 or -1, and    the red slots are determined to  insure  there are no runs of length greater than $m$.  Hence there are at least
$t-\lfloor \frac{t}{m}\rfloor -1$ black slots,  and since the worst case up to cyclic conjugacy  is that all $t$ of the cyclic conjugates are distinct and in $\mathcal{B}_{2t,m}$ we have that the conjugacy classes of these elements  satisfy,
$|B_{2t,m}| \geq \frac{2^{t-\lfloor \frac{t}{m}\rfloor -1}}{t}$. In fact, we have the following.

\begin{thm}
\label{thm:lowlyingbd}
\leavevmode
\begin{enumerate}

\item  $|L_{2t,m}| \geq |B_{2t,m}| \geq \frac{2^{t-\frac{t}{m} -1}}{t}$, for $m\geq 2$.\\
\item For $m\geq3$, $|L_{2t,m}^p|\sim|L_{2t,m}|$, as $t\to\infty$.
\item There exists $t_0 >0$ so that $|L^{p}_{2t,m}|  
\geq \frac{1}{2} \left(\frac{2^{t-\frac{t}{m} -1}}{t}\right)$, for $t \geq t_0$ and $m\geq 3$.
\item  
There exists  $t_0 >0$ so that $|L^{p}_{\leq 2t,m}|  
\geq \frac{1}{4} \displaystyle\sum_{s=t_0}^{t} \frac{2^{s-\frac{s}{m}}}{s}$, for $t \geq t_0$ and $m\geq3$.
\end{enumerate}
\end{thm}

\begin{proof}  Item (1) was proven in the discussion before the theorem.   To prove  item (2),  we first use  Lemma  
\ref{lem: {W}_{2t}^{np}} to bound the nonprimitive low lying growth rate,
\begin{displaymath}
|L^{np}_{2t,m}| \leq  |W_{2t}^{np}|\leq\frac{1}{2}t2^{t/2}\ 
\end{displaymath}
hence,

\begin{equation}\label{eq: lower bound inequality}
  1 \geq  \frac{|L^{p}_{2t,m}|}{|L_{2t,m}|}  
  =  1-\frac{|L^{np}_{2t,m}|}{|L_{2t,m}|}  
  \geq   1-\frac{\frac{1}{2}t2^{t/2}}{ \frac{2^{t-\frac{t}{m}-1}}{t}}
  \geq    1-\frac{\frac{1}{2}t2^{t/2}}{ \frac{2^{t-\frac{t}{3} -1}}{t}}
\end{equation}
where we have used $m \geq 3$  in the right-hand inequality of expression (\ref{eq: lower bound inequality}). 

Noting that the lower bound in expression (\ref{eq: lower bound inequality}) does not depend on $m$,  item (3) follows from item (1) and  by choosing  $t_0$ large enough so that 
\begin{displaymath}
\frac{|L^{p}_{2t,m}|}{|L_{2t,m}|} \geq  \frac{1}{2}.
\end{displaymath}
 . 

Finally to prove  item (4),

\begin{eqnarray*}
|L^{p}_{\leq 2t,m}|=\sum_{s=1}^{t} |L^{p}_{2s,m}|
=\sum_{s=1}^{t_0} |L^{p}_{2s,m}| + \sum_{s=t_0}^{t} |L^{p}_{2s,m}| 
&\geq&  \sum_{s=1}^{t_0} |L^{p}_{2s,m}|+
 \frac{1}{2} \sum_{s=t_0}^{t}
 \left(\frac{2^{s-\frac{s}{m} -1}}{s}\right)\\
 &\geq&  \frac{1}{2} \sum_{s=t_0}^{t}
 \left(\frac{2^{s-\frac{s}{m} -1}}{s}\right)\\
 &=&\frac{1}{4} \sum_{s=t_0}^{t} \frac{2^{s-\frac{s}{m}}}{s}.
  \end{eqnarray*}

\end{proof}

\begin{rem}
Since our eventual goal is to prove Theorem \ref{thm:main} and Corollary \ref{cor:main}  it is critical that $t_0$ in the above theorem does not depend on $m$.
\end{rem}

\subsection{Low lying reciprocal words}

In this section we count the low lying reciprocal words.  Recall that \[\mathcal{N}_{4t}=\{[a,\gamma]: \gamma \,\textrm{a $(bb)$-word of length} \,2t-1\}\] where $[a,\gamma]$ is the group commutator of $a$ and $\gamma$.
We define $\pi : \mathcal{N}_{4t} \rightarrow  R_{4t}$, to be the map taking elements in $\mathcal{N}_{4t}$ to its conjugacy class.  From Section \ref{countingbinarywords} we know  that there are exactly two conjugacy class representatives in normal form for reciprocal words, hence $\pi$ is a surjective 2-1 mapping. 
 
\begin{definition} Fix an integer, $t>0$. A  composition of $t$ is  an ordered sequence 
of positive integers, $(k_1,...,k_l)$ which   sums  to $t$.  The ${k_i}'s$ are called the parts of the composition. The set of all compositions of $t$ is denoted $C_{t}$. Compositions of $t$ having parts bounded by a fixed positive integer $m$ are denoted,  $C_{t,m}$. 
\end{definition}

Next, define $g: \mathcal{N}_{4t}\to C_t$, where $g([a,b^{\epsilon_1}ab^{\epsilon_2}\dots ab^{\epsilon_t}])=(k_1,\dots,k_l)$.  Here $(k_1,\dots,k_l)$ is the ordered sequence of  lengths of (+1) and (-1)-runs starting from the left in the $\epsilon_i$'s.  For example, if $\omega=[a,b^{-1}ab^{-1}ab^{-1}ab^{1}ab^{1}ab^{-1}ab^{1}]\in \mathcal{N}_{28}$, then $g(\omega)=(3,2,1,1)\in C_7$.  We remark here that $g$ is a surjective 2-1 mapping.  Namely, suppose $\omega=[a,b^{\epsilon_1}ab^{\epsilon_2}\dots ab^{\epsilon_t}] \in \mathcal{N}_{4t}$ so that $g(\omega)=(k_1,\dots,k_l)\in C_t$. Then there exists exactly one other element in $\mathcal{N}_{4t}$ whose image under $g$ is $(k_1,\dots,k_l)$, namely, $[a,b^{-\epsilon_1}ab^{-\epsilon_2}\dots ab^{-\epsilon_t}]$.




\label{fig:lowlyingreciprocal}
\begin{minipage}{.35\linewidth}
\[\begin{tikzcd}[column sep=large, row sep=large]
\mathcal{N}_{4t}\arrow[r, "g"] \arrow[d, "\pi"']
& C_t\\
R_{4t} &\\
\end{tikzcd}
\]
\end{minipage}
\hspace{1cm}
\begin{minipage}{.35\linewidth}

\[\begin{tikzcd}[column sep=large, row sep=large]
\mathcal{L}_{4t,m}\cap\mathcal{N}_{4t}\arrow[r, "g_m"] \arrow[d, "\pi_m"']
& C_{t,m}\\
L_{4t,m}\cap R_{4t}  \arrow[ur, dashed, "\Phi_m"'] &\\
\end{tikzcd}
\]
\end{minipage}

The set of reciprocal words is filtered by low lying reciprocal words.  This is because the set of all words $\mathcal{W}$ has filtration,

\begin{equation*}\mathcal{L}_{4t,1}\subset \mathcal{L}_{4t,2}\subset\dots \subset\mathcal{L}_{4t,m}\subset\dots
\end{equation*}

This induces a filtration of each of the spaces in the left diagram to yield restricted mappings in the right diagram.  Noting that the maps $g_m$ and $\pi_m$ are surjective 2-1 maps, even though the preimage of a point in $L_{4t,m}\cap R_{4t}$ is generically not the same as the preimage of a point in $C_{t,m}$, we have, $2|L_{4t,m}\cap R_{4t}|=|\mathcal{L}_{4t,m}\cap\mathcal{N}_{4t}|=2|C_{t,m}|.$  Thus we have,

\begin{thm}
For each positive integer $m\geq2$, there is a bijection $\Phi_m$ from $L_{4t,m}\cap R_{4t}$ to $C_{t,m}$, given by the right diagram.
\label{bijection}
\end{thm}

We have now reduced the problem to counting $C_{t,m}$, whose computation involves using the recursion relation: $|C_{t,m}|=\sum_{i=1}^m |C_{t-i,m}|$.  The combinatorial analysis solving this problem follows from Theorem 2 of \cite{Dr}, where $C_{t,m}$ is denoted by $F_{t+1}^{(m)}$ in the article.  We have
\begin{equation}
\label{boundedcompositions}
|C_{t,m}|=
rnd\left(\frac{\alpha_m -1}{2+(m+1)(\alpha_m -2)}\alpha_m^{t}\right)
\end{equation}
where $rnd(x)=\lfloor{x+\frac{1}{2}}\rfloor$ and $\alpha_m$ is the unique positive root of $z^m-z^{m-1}-\dots-1=0$.  We remark that the coefficient $\frac{\alpha_m -1}{2+(m+1)(\alpha_m -2)}$ in equation (\ref{boundedcompositions}) only depends on $m$ and thus we denote it by $d_m$.  We note that  $2(1-2^{-m}) \leq  \alpha_m < 2$ and the $\alpha_m$ are increasing as $m$ increases. For the details see \cite{Dr}.

We have proven the following theorem.

\begin{thm}
\label{thm:lowlyingrec}
\leavevmode
\begin{enumerate}
\item
$d_m\alpha_m^{t}-\frac{1}{2}\leq|L_{4t,m}\cap R_{4t}|\leq d_m\alpha_m^{t}+\frac{1}{2}, \textrm{for all}\,\,t\geq1$\\

\item
$|L_{4t,m}\cap R_{4t}|\sim d_m\alpha_m^{t}, \textrm{as}\,\,t\to\infty$

\end{enumerate}
\end{thm}

\begin{cor}
\label{cor:lowlyingrec}
\begin{equation*}
|L_{\leq4t,m}\cap R_{\leq4t}|\sim \left(\dfrac{\alpha_m}{2+(m+1)(\alpha_m-2)}\right)\alpha_m^{t}, \textrm{as}\,\,t\to\infty
\end{equation*}
\end{cor}

\begin{proof}

First, note that $$|L_{\leq4t,m}\cap R_{\leq4t}|=\displaystyle\sum_{n=1}^t|L_{4n,m}\cap R_{4n}|.$$  Applying item (1) of Theorem \ref{thm:lowlyingrec} gives us,

\begin{equation*}
\displaystyle\sum_{n=1}^t\left(d_m\alpha_m^n-\frac{1}{2}\right)\leq |L_{\leq4t,m}\cap R_{\leq4t}|\leq \displaystyle\sum_{n=1}^t\left(d_m\alpha_m^n+\frac{1}{2}\right).
\end{equation*}

Simplifying, we have

\begin{equation*}
\left(\dfrac{\alpha_m}{2+(m+1)(\alpha_m-2)}\right)(\alpha_m^{t}-1)-\frac{t}{2} \leq |L_{\leq4t,m}\cap R_{\leq4t}| \leq \left(\dfrac{\alpha_m}{2+(m+1)(\alpha_m-2)}\right)(\alpha_m^{t}-1)+\frac{t}{2}.
\end{equation*}

The result follows by dividing by $\left(\dfrac{\alpha_m}{2+(m+1)(\alpha_m-2)}\right)\alpha_m^{t}$ and letting $t\to\infty$.

\end{proof}


\section{Representation and  Geodesic excursion into the cusp}
\label{sec:representations and geodesic excursion}

 Consider the group  $G=\mathbb{Z}_2\ast \mathbb{Z}_3$ with 
 generators   $a$  and   $b$ of the first and second factors
 resp.  This  group is isomorphic to the modular group. 
 We consider the following representation of $G$:
 $a \mapsto A$ and $b \mapsto B$ where,

$$
A=\begin{pmatrix} 
     0 & -1   \\
     1 & 0 
\end{pmatrix} \text{ and } B=\begin{pmatrix} 
     1 & -1    \\
     1 & 0 
\end{pmatrix}.
$$

This is a discrete, faithful representation with image $PSL(2,\mathbb{Z}).$ Let $S=\mathbb{H}/PSL(2,\mathbb{Z})$ be the associated orbifold surface.  It follows that $S$ is a generalized pair of pants.  In particular, $S$ has zero genus and signature $(2,3,\infty)$.

It is well known that  if an  orbifold surface has a cusp then it has an embedded cusp of area one  and boundary a horocycle segment of length one.  A closed geodesic that wanders into (and hence out of a cusp) has a maximal depth in which it enters. More precisely,
let $\gamma$ be  a closed geodesic on  the modular orbifold and 
$\mathcal{C}$ the  cusp of area one. The closed geodesic may wander in and out of the cusp a number of times, and each time it enters and exits the cusp we call this  an {\it excursion} of $\gamma$. The {\it depth} of an excursion is the furthest distance into the cusp the excursion goes. 

\begin{lem} \label{lem:wandering deep in cusp}
 Let  $\gamma$ be a closed geodesic on $S$. 
\begin{enumerate}
\item   An excursion
of $\gamma$ winds $k \geq 2$ times around the cusp if and only if 
the depth of the excursion is strictly  between $\log \frac{k}{2}$ and  $\log \frac{k+1}{2}$
\item $\gamma$ is contained in the $m$-thick part of $S$ if and only if some, and hence any,  representative $g \in PSL(2, \mathbb{Z})$  of $\gamma$  is an  $m$-low lying word. 

\end{enumerate}

\end{lem}


\begin{proof} Let $\widetilde{\gamma}$ be a  lift  of $\gamma$ normalized so that its endpoints at infinity are $-r$ and $r$ and the parabolic associated to the cusp normalized to be $f(z)=z+1$. The depth for this excursion  into the cusp  is $\log r$. Now if the excursion winds around the cusp $k$-times 
then $f^{k}(\tilde{\gamma}) \bigcap \tilde{\gamma} 
\neq \emptyset$ and $f^{k+1}(\tilde{\gamma}) \bigcap \tilde{\gamma}  = \emptyset$. That is, 
$\frac{k}{2}<r<\frac{k+1}{2}$. See Figure \ref{fig: hyperbolicplane}. Note that equality is not included  as that would violate the fact that two hyperbolic elements in a Fuchsian group can not share a unique fixed point.  Equivalently, 
$\log \frac{k}{2} < \log r < \log \frac{k+1}{2}$. These steps are reversible. Hence we have proven item (1).

To prove item (2),   
 the word   $g$   written as  a product of the generators in normal form   is the product of the inverse conjugate parabolic elements
 $AB$ and $AB^{-1}$. Suppose $g$ is $m$-low lying. Then the longest run of  $AB$  or $AB^{-1}$ (considered cyclically) is at most $m$. Now, a  run in the word $g$, say $k$, corresponds to $\gamma$ winding around the cusp $k$ times. By item (1), we know that the depth of this excursion is at most $\log \frac{k+1}{2} \leq \log \frac{m+1}{2}$. Thus $\gamma$ is contained in the $m$-thick part of $S$. For the converse, if $\gamma$ is in the $m$-thick part then item (1) again guarantees that there is no run of $AB$ or $AB^{-1}$ longer than $m$. Therefore $g$ is an $m$-low lying word. 
 \end{proof}

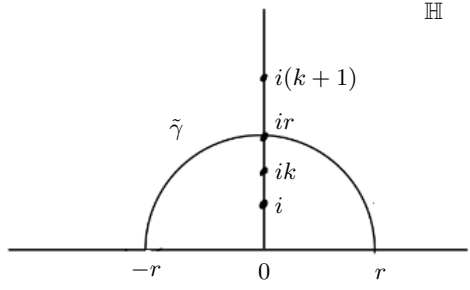
\begin{figure}
\begin{center}
\begin{overpic}[scale=.7]{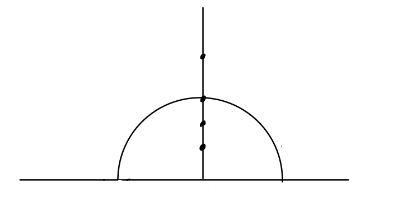}
\put(34,29){\footnotesize{$\tilde{\gamma}$}}
\put(50,3){\footnotesize{$0$}}
\put(27,3.5){\footnotesize{$-r$}}
\put(71,3){\footnotesize{$r$}}
\put(80,50){\footnotesize{$\mathbb{H}$}}
\put(53,15){\footnotesize{$i$}}
\put(53,21){\footnotesize{$ik$}}
\put(53,30){\footnotesize{$ir$}}
\put(53,38){\footnotesize{$i(k+1)$}}
\end{overpic}
\end{center}
\caption{The hyperbolic plane}
\label{fig: hyperbolicplane}
\end{figure}


\section{All Together Now: The Proof of Theorem \ref{thm:main}} \label{sec:all together}

In this section we put together the work of the previous sections to prove Theorem \ref{thm:main}.  The closed geodesics on $S$ correspond to conjugacy classes of hyperbolic elements in $PSL(2,\mathbb{Z})$. Similarly, reciprocal closed geodesics correspond to hyperbolic elements whose axes pass through an order two fixed point.  Finally, low lying closed geodesics correspond to conjugacy classes of low lying words, as in Lemma \ref{lem:wandering deep in cusp}.  Using the notation from the previous sections, we have the following correspondence between geodesics and conjugacy classes of words in the group.

\begin{eqnarray*}
\{\gamma \text{ a primitive reciprocal geodesic with }
 |\gamma| \leq 2t\}&\longleftrightarrow& R_{\leq 2t}^p\\
\{\gamma \text{ a primitive   reciprocal geodesic in $S_{m\geq2}$\,with }
 |\gamma| \leq 2t\}&\longleftrightarrow& L_{\leq 2t,m}^p\cap R_{\leq 2t}^p\\
\{\gamma \text{ a primitive closed geodesic with }
 |\gamma| \leq 2t\}&\longleftrightarrow& \{[w]\in W^p_{\leq 2t}: ||[w]||>2\}\\
\{\gamma \text{ a primitive closed geodesic in $S_{m\geq3}$ with }
 |\gamma| \leq 2t\}&\longleftrightarrow& \{[w]\in L^p_{\leq 2t,m}: ||[w]||>2\}
\end{eqnarray*}

\begin{proof}[Proof of Theorem \ref{thm:main}]
To prove item (1), we first note that 

\[|R_{\leq 2t}|=\displaystyle\sum_{n=1}^{\lfloor \frac{t}{2}\rfloor}|R_{4n}|=2^{\lfloor \frac{t}{2}\rfloor}-1\] where the last equality follows  as in the proof of the second part of Lemma \ref{lem:reciprocal conj sizes}. Using the fact that $|R_{2t}|\leq\frac{1}{2}2^{\frac{t}{2}}$, we can establish $|R^p_{\leq 2t}|\sim|R_{\leq 2t}|$ in a similar way to what was done for length $4t$.

To prove  item (2), note that \[|L_{\leq 2t,m}\cap R_{\leq 2t}|=\displaystyle\sum_{n=1}^{\lfloor \frac{t}{2}\rfloor}|L_{4n,m}\cap R_{4n}|.\] As in the proof of Corollary \ref{cor:lowlyingrec}, applying item (1) from Theorem \ref{thm:lowlyingrec} yields 

$$|L_{\leq 2t,m}\cap R_{\leq 2t}|\sim \left(\dfrac{\alpha_m}{2+(m+1)(\alpha_m-2)}\right)\alpha_m^{\lfloor\frac{t}{2}\rfloor}.$$ Using the fact that $|L_{2t,m}\cap R_{2t}|\leq rnd(d_m\alpha_m^{\frac{t}{2}})$ it is not difficult to show \newline $|L_{\leq 2t,m}^p\cap R_{\leq 2t}^p|\sim|L_{\leq 2t,m}\cap R_{\leq 2t}|$.

Item (3) follows from Theorem \ref{thm:Wasymptotics primitive} and the fact that the primitive hyperbolic conjugacy classes have the same growth as all primitive conjugacy classes.  

Lastly, item (4) follows from Theorem  \ref{thm:lowlyingbd}, item (4), and an application of  the Stolz-Cesaro Theorem to show 

$$\displaystyle\sum_{s=t_0}^{t} \frac{2^{s-\frac{s}{m}}}{s}\sim\frac{2}{2-2^{\frac{1}{m}}}\left(\frac{2^{t-\frac{t}{m}}}{t}\right).$$

\end{proof}


\begin{thebibliography}{5}


\bibitem{Boca} F. Boca, V. Pa\c sol, A. Popa, \& A. Zaharescu, {\it Pair correlation of angles between reciprocal geodesics on the modular surface.} Algebra Number Theory 8 (2014), no. 4, 999-1035.

\bibitem{B-K2} J. Bourgain \& A. Kontorovich, {\it Beyond
expansion II: low-lying fundamental geodesics.} J. Eur. Math. Soc. (JEMS) 19 (2017), no. 5, 1331-1359.

\bibitem{B-K3} J. Bourgain \& A. Kontorovich, {\it Beyond
expansion III: Reciprocal geodesics.} Duke Math J. 168 (2019), no.18, 3413-3435.

\bibitem{Bus} P. Buser, {\it Geometry and spectra of compact Riemann surfaces.} Progress in Mathematics, 106. Birkh\"{a}user Boston,  Inc., Boston, MA 1992. xiv+454 pp.

\bibitem{CalLou} D. Calegari \& J. Louwsma, {\it Immersed surfaces in the modular orbifold.} Proc. Amer. Math. Soc. 139 (2011), no. 7, 2295-2308.


\bibitem{Dr} G. Dresden \& Z. Du, {\it A Simplified Binet Formula for $k$-Generalized Fibonacci Numbers.} J. Integer Seq. 17 (2014), no. 4, Article 14.4.7, 9 pp.

\bibitem{Er} V. Erlandsson, {\it A remark on the word length in surface groups.} Trans. Amer. Math. Soc. 372 (2019), no.1, 441-445.

\bibitem{ErPaSo} V. Erlandsson, H. Parlier, \& J. Souto, {\it Counting curves, and the stable length of currents.} J. Eur. Math. Soc. (JEMS) 22 (2020), no. 6, 1675-1702.

\bibitem{ErSo} V. Erlandsson \& J. Souto, {\it Couting curves in hyperbolic surfaces.} Geom. Funct. Anal. 26 (2016), no. 3, 729-777.


\bibitem{GubSap} V. Guba \& M. Sapir, {\it On the conjugacy growth functions of groups.} Illinois J. Math. 54 (2010), no. 1, 301-313.


\bibitem{LyndSch} R. Lyndon \& P. Schupp, {\it Combinatorial group theory.} Reprint of the 1977 edition. Classics in Mathematics. Springer-Verlag, Berlin, 2001. xiv+339 pp. ISBN: 3-540-41158-5

 \bibitem{Mag} W. Magnus, A. Karrass, \& D. Solitar, {\it Combinatorial group theory.} Presentations of groups in terms of generators and relations. Second revised edition. Dover Publications, Inc., New York, 1976. xii+444 pp.

\bibitem{Mirz} M. Mirzakhani, {\it Growth of the number of simple closed geodesics on hyperbolic surfaces.} Ann. of Math. (2) 168 (2008), no. 1, 97-125.

\bibitem{Mur} M. Mure\c{s}an, {\it A concrete approach to classical analysis.} CMS Books in Mathematics/Ouvrages de Mathématiques de la SMC. Springer, New York, 2009. xviii+433

\bibitem{Park} P.S. Park, {\it Conjugacy growth of commutators.} J. Algebra 526 (2019), 423-458.

 \bibitem{Ri} I. Rivin, {\it Growth in free groups (and other stories) -- twelve years later.} Illinois J. Math. 54 (2010), no. 1, 327-370.


  \bibitem{Sar} P. Sarnak, {\it Reciprocal Geodesics}. Analytic number theory, 217-237, Clay Math. Proc., 7, Amer. Math. Soc., Providence, RI, 2007.

\bibitem{Tra} C. Traina, {\it The conjugacy problem of the modular group and the class number of real quadratic number fields.} J. Number Theory 21 (1985), no. 2, 176-184.

 
  
 \end{thebibliography}
\end{document}